\input louC.sty

\documentclass[a4paper,draft]{amsproc}
\usepackage{amsrefs}
\usepackage{amsmath}
\usepackage{mathrsfs}
\usepackage{amssymb}
\usepackage{amscd} %% Package for commutative diagrams
\theoremstyle{plain}
 \newtheorem{thm}{\textbf{Theorem}}[section]
 \newtheorem{prop}{\textbf{Proposition}}[section]
 \newtheorem{lem}{\textbf{Lemma}}[section]
 \newtheorem{cor}{\textbf{Corollary}}[section]
\theoremstyle{definition}
 \newtheorem{exm}{\textbf{Example}}[section]
 
\theoremstyle{remark}
 \newtheorem{rem}{\textbf{Remark}}[section]
 \numberwithin{equation}{section}

\renewcommand{\leq}{\leqslant}
\renewcommand{\geq}{\geqslant}

\title[Composition Operators on the Bidisk]{Bounded Composition Operators and Multipliers of Some Reproducing Kernel Hilbert Spaces on the Bidisk}

\author[Chu]{\bfseries Cheng Chu}

\address{
Department of Mathematics \\ % \hfill (Received 00 00 2010)\\
Vanderbilt University  \\ %\hfill (Revised  00 00 2010)\\
Nashville, Tennessee \\
USA}
\email{cheng.chu@vanderbilt.edu}

\begin{document}

%{\begin{flushleft}\baselineskip9pt\scriptsize
%PUBLICATIONS DE L'INSTITUT MATH\'EMATIQUE\newline
%Nouvelle s\'erie, tome 87(101) (2010), od--do \hfill DOI:
%\end{flushleft}}
\vspace{18mm}
\setcounter{page}{1}
\thispagestyle{empty}

\begin{abstract}
We study the boundedness of composition operators on the bidisk using reproducing kernels. We show that a composition operator is bounded on the Hardy space $H^2(\DD^2)$ if some associated function is a positive kernel. This positivity condition naturally leads to the study of the sub-Hardy Hilbert spaces of the bidisk, which are analogs of de Branges-Rovnyak spaces on the unit disk. We discuss multipliers of those spaces and obtain some classes of bounded composition operators on the bidisk.
\end{abstract}

\maketitle

\section{Introduction}  %% Please avoid complicated formulas in titles
Let $\DD$ denote the open unit disk in $\CC$ with boundary $\TT$. The bidisk $\DD^2$ and the torus $\TT^2$ are the subsets of $\CC^2$ which are Cartesian products of two copies $\DD$ and $\TT$, respectively. The Hardy space $H^2(\DD)$ is the closure of the analytic polynomials in $L^2(\TT)$ and the Hardy space $H^2(\DD^2)$ (or $H^2$) is the closure of the analytic polynomials in $L^2(\TT^2, d\Gs)$ (or $L^2(\TT^2)$), where $d\Gs$ is the normalized Haar measure on $\TT^2$. $H^\infty(\DD^2)$ is the space of bounded analytic functions on $\DD^2$ with norm
$$||f||_\infty=\sup_{(z_1,z_2)\in\DD^2}|f(z_1,z_2)|.$$

For a bounded domain $\GO\subset\CC^d$ ($\GO=\DD$ or $\DD^2$), the composition operator $C_\varphi$ on $H^{2}(\GO)$ is defined by $C_{\varphi}f=f\circ\varphi$, for an analytic self-map $\varphi$ of $\GO$. One of the fundamental problems is to classify the mappings $\varphi$ which induce bounded operators $C_\varphi$, and this problem has been studied on various analytic function spaces. Littlewood's famous Subordination Principle implies that each composition operator is bounded on $H^2(\DD)$ (e.g. \cite{sha93}*{p. 31}), and it is known that the same result does not hold in higher dimensions. In the Hardy space or weighted Bergman spaces of the unit ball in $\CC^n$, the boundedness of $C_\varphi$ is characterized using Carleson measure (e.g. \cite{cm95}*{Section 3.5}).

Another natural extension of complex analysis to the multivariable case is the bidisk, or polydisk in the more general higher dimensional setting. It is known that composition operators $C_B$ on $H^2(\DD^2)$ are not always bounded \cite{sish}; one such example is to take $B(z_1,z_2)=(z_1,z_1)$. On the other hand, one can easily show that $C_B$ is bounded on $H^2(\DD^2)$ in some special cases (Example \ref{ex0}). In this paper, we find new nontrivial classes of bounded composition operators on $H^2(\DD^2)$.

In \cite{ju07}, Jury reproved the boundedness of composition operators on $H^2(\DD)$ using only reproducing kernels. We adapt this Hilbert space method to study the two-variable case and find a sufficient condition for the boundedness of a composition operator on $H^2(\DD^2)$ (Theorem \ref{M}). To obtain functions that satisfy that condition, it is equivalent to find multipliers of the so-called sub-Hardy Hilbert spaces on the bidisk (see the definition in Section 2), which are analogues of de Branges-Rovnyak spaces on the disk. We discuss the multipliers of sub-Hardy Hilbert spaces in Section 4 and get some new classes of bounded composition operators (Theorem \ref{com}). Several concrete examples are presented in Section 5.

\section{Preliminaries}\label{Pre}
In this section, we first present some basic theory of reproducing kernel Hilbert spaces. For more information about reproducing kernels and their associated Hilbert spaces, see \cite{aro50} and \cite{pau16}.

Let $X\subset \CC^d$. We say a function $K: X\times X\to\CC$ is a positive kernel on $X$ if it is self-adjoint ($K(x,y)=\ol{K(y,x)}$), and for all finite sets $\{\Gl_1,\Gl_2,\dots, \Gl_{m}\}\subset X$, the matrix $(K(\Gl_i,\Gl_j))_{i,j=1}^m$ is positive semi-definite.

Here are the usual ways to construct new positive kernels from old ones (see for example \cite{aro50}).
\begin{prop}\label{prop}
Let $K_1, K_2$ be positive kernels on $X$. Then
\begin{enumerate}
\item $K_1+K_2$ is a positive kernel.
\item $K_1\cdot K_2$ is a positive kernel.
\item If $f: X\to\CC$ is a function, then $\overline{f(w)}f(z)K_1(z,w)$ is a positive kernel.
\end{enumerate}
\end{prop}

A reproducing kernel Hilbert space $\mathcal{H}$ on $X$ is a Hilbert space of complex valued functions on $X$ such that every point evaluation is a continuous linear functional. Thus there exists an element $K_w\in\mathcal{H}$ such that for each $f\in\mathcal{H}$, $$\langle f, K_w\rangle_{\mathcal{H}} =f(w).$$
Since $K_w(z)=\langle K_w, K_z\rangle_{\mathcal{H}}$, $K$ can be regarded as a function on $X\times X$ and we write $K(z,w)= K_w(z)$. Such $K$ is a positive kernel and the Hilbert space $\mathcal{H}$ with reproducing kernel $K$ is denoted by $\mathcal{H}(K)$.

The following theorem, due to Moore, shows that there is a one-to-one correspondence between reproducing kernel Hilbert spaces and positive kernels (see for example \cite{ampi}*{Theorem 2.23}).
\begin{thm}\label{Moore}
Let $X\subset \CC^d$ and let $K: X\times X\to\CC$ be a positive kernel. Then there exists a unique reproducing kernel Hilbert space $\mathcal{H}(K)$ whose reproducing kernel is $K$.
\end{thm}

We also need the next theorem that characterizes the functions that belong to a reproducing kernel Hilbert space in terms of the reproducing kernel.
\begin{thm}\cite{pau16}*{Theorem 3.11}\label{or}
Let $\mathcal{H}(K)$ be a reproducing kernel Hilbert space on $X$ and let $f: X\to\CC$ be a function. Then $f\in \mathcal{H}(K)$ with $||f||_{\mathcal{H}(K)}\leq c$ if and only if
$$
c^2K(z,w)-\overline{f(w)}f(z)
$$
is a positive kernel.
\end{thm}

A function $\varphi: X\to\CC$ is called a multiplier of $\mathcal{H}(K)$ on $X$ if $\varphi f\in \mathcal{H}(K)$ whenever $f\in \mathcal{H}(K)$. If $\varphi$ is a multiplier of $\mathcal{H}(K)$, let $M_\varphi: f\longmapsto \varphi f$ be the multiplication operator on $\mathcal{H}(K)$. In this case, it is well-known that the kernel functions are eigenvectors for the adjoints of multiplication operators: $M_{\varphi}^{*}K_z=\overline{\varphi(z)}K_z$, and, as a consequence (see for example \cite{pau16}*{Chapter 5.7})
\beq\label{norm}||M_{\varphi}||\geq \sup_{z\in X}|\varphi(z)|.\eeq

The following theorem characterizes multipliers of reproducing kernel Hilbert spaces.
\begin{thm}\cite{ampi}*{Corollary 2.37}\label{m}
Let $\mathcal{H}(K)$ be a reproducing kernel Hilbert space on $X$, and let $\varphi: X\to\CC$ be a function. Then $\varphi$ is a multiplier of $\mathcal{H}(K)$ with multiplier norm at most $\Gd$ if and only if
$$
(\Gd^2-\varphi(z)\overline{\varphi(w)})\cdot K(z,w)
$$
is a positive kernel. If $\Gd\leq 1$, then $\varphi$ is called a contractive multiplier of $\mathcal{H}(K)$.
\end{thm}

As a corollary of Theorem \ref{m}, we have
\begin{cor}\label{in}
Let $\mathcal{H}(K_1)$ and $\mathcal{H}(K_2)$ be reproducing kernel Hilbert spaces on $X$. Then $\mathcal{H}(K_1) \subset \mathcal{H}(K_2)$ if and only if there is some constant $\Gd>0$ such that
$$
K_2(z,w)-\frac{1}{\Gd^2}K_1(z,w)
$$
is a positive kernel.
\end{cor}

We shall also use Toeplitz and Hankel operators. Let $P$ be the orthogonal projection from $L^2 (\TT^2)$ onto $ H^2(\DD^2)$. The Toepltiz operator with symbol $\varphi\in L^\infty(\TT^2)$ is defined by $$T_\varphi(h)=P(\varphi h),$$ for all $h\in H^2(\DD^2)$.
The Hankel operator with symbol $\varphi\in L^\infty(\TT^2)$ is defined by $$H_\varphi(h)=(I-P)(\varphi h),$$ for all $h\in H^2(\DD^2)$.

The following identity is well-known and easily established
\beq\label{han}
T_{fg}=T_fT_g-H_{\bar{f}}^*H_g,
\eeq
for $f,g \in L^\infty(\TT^2)$.

For a point $z\in\DD^2$, we use $z=(z_1, z_2)$ to denote the coordinates of $z$.
Define the coordinate maps $$P_1: \DD^2\to \{0\}\times\DD, (z_1,z_2)\longmapsto(0, z_2)$$ and
$$P_2: \DD^2\to\DD\times\{0\}, (z_1,z_2)\longmapsto(z_1, 0).$$
We have the following formula for the two backward shift operators
\beq\label{bs}
T_{\bz_j} f=\frac{f-f\circ P_j}{z_j},\q j=1,2,
\eeq
for any $f\in H^2(\DD^2)$.

\section{Bounded Composition Operator on $H^2(\DD^2)$}\label{3}
The following theorem is the main tool to find bounded composition operators.
\begin{thm}\label{M}
Let $B=(\phi, \psi)$ be an analytic map from $\DD^2$ to $\DD^2$, where $\phi$ and $\psi$ are in $H^\infty(\DD^2)$ and bounded by $1$.
Define a function $R$ on $\DD^2\times\DD^2$ as
\beq
R(z,w)=\frac{1-\overline{\phi(w)}{\phi(z)}}{1-\bar{w_1}z_1}\cdot \frac{1-\overline{\psi(w)}{\psi(z)}}{1-\bar{w_2}z_2}
\eeq

If $R$ is a positive kernel, then $C_B$ is a bounded composition operator on $H^2(\DD^2)$ and $$||C_B||\leq \left( \frac{1+|\phi(0)|}{1-|\phi(0)|} \right)^{1\over 2}\cdot \left( \frac{1+|\psi(0)|}{1-|\psi(0)|} \right)^{1\over 2}.$$
\end{thm}

\begin{proof}

The reproducing kernel of $H^2(\DD^2)$ is the the Szeg\H{o} kernel $$k_w(z)=\frac{1}{(1-\bar{w_1}z_1)(1-\bar{w_2}z_2)}.$$
We densely define an operator $C_B^*$ on $H^2(\DD^2)$ by $C_B^*k_z=k_{B(z)}$. If $f$ and $f\circ B$ are in $H^2(\DD^2)$, then
$$
\langle C_Bf, k_z\rangle_{H^2}=\langle f\circ B, k_z\rangle_{H^2}=f(B(z))=\langle f, k_{B(z)}\rangle_{H^2}=\langle f,C_B^* k_z\rangle_{H^2}.
$$
So $C_B^*$ is the formal adjoint of the composition operator $C_B$, thus it is sufficient to prove $C_B^*$ is bounded. Using
$$\langle C_B^{*}k_w, C_B^{*}k_z\rangle_{H^2} =\langle k_{B(w)}, k_{B(z)}\rangle_{H^2} = \frac{1}{(1-\overline{\phi(w)}\phi(z))(1-\overline{\psi(w)}\psi(z))},$$

we have \beq\label{*}\langle  k_w, k_z\rangle_{H^2}=\langle C_B^{*}k_w, C_B^{*}k_z\rangle_{H^2} \cdot R(z,w).\eeq

Since $R$ is a positive kernel, by Theorem \ref{Moore}, there exists a reproducing kernel Hilbert space $\mathcal{H}(R)$, such that $R(z,w)=\langle R_w,R_z\rangle_{\mathcal{H}(R)},$ where $R_w$ is the reproducing kernel of $\mathcal{H}(R)$ at $w$.

Let $$F(z)=\frac{R_0(z)}{||R_0||_{\mathcal{H}(R)}}.$$

Then $||F(z)||_{\mathcal{H}(R)}=1$ and by Theorem \ref{or}, $R(z,w)-\overline{F(w)}F(z)$ is a positive kernel. Multiply it by the positive kernel $\langle  C_B^{*}k_w, C_B^{*}k_z \rangle_{H^2}$ and using \eqref{*} and Proposition \ref{prop}(2), we get
\beq\label{1}
\langle  k_w, k_z\rangle_{H^2}-\overline{F(w)}F(z)\langle  C_B^{*}k_w, C_B^{*}k_z \rangle_{H^2}
\eeq
is a positive kernel.

Notice that
\beq\label{F}
F(z)=\frac{\langle R_0,R_z\rangle_{\mathcal{H}(R)}}{||R_{0}||}= \frac{R(z,0)}{\sqrt{R(0,0)}}=\frac{1-\overline{\phi(0)}\phi(z)}{\sqrt{1-|\phi(0)|^2}}\cdot\frac{1-\overline{\psi(0)}\psi(z)}{\sqrt{1-|\psi(0)|^2}}.
\eeq
We have $F\in H^{\infty}(\DD^2)$, and then $M_{F}^{*}k_z=\overline{F(z)}k_z$.

Therefore
$$
\overline{F(w)}F(z)\langle  C_B^{*}k_w, C_B^{*}k_z \rangle_{H^2}=\langle C_B^{*}M_{F}^{*}k_w, C_B^{*}M_{F}^{*}k_z \rangle_{H^2},
$$
and by \eqref{1},
$$
\langle  k_w, k_z\rangle_{H^2}-\langle  C_B^{*}M_{F}^{*}k_w, C_B^{*}M_{F}^{*}k_z \rangle_{H^2}
$$
is a positive kernel.

For any $n$ distinct points $z^{(1)},\cdots,z^{(n)}$ in $\DD^2$ and complex numbers $c_1,\cdots, c_n$, define $h=\sum_{j=1}^{n}c_j k_{z^{(j)}}$. We then have

\begin{align*}
0&\leq \sum_{i,j=1}^n c_i\bar{c_j} \langle  k_{z^{(i)}}, k_{z^{(j)}}\rangle_{H^2}-\sum_{i,j=1}^n c_i\bar{c_j} \langle  C_B^{*}M_{F}^{*}k_{z^{(i)}}, C_B^{*}M_{F}^{*}k_{z^{(j)}}\rangle_{H^2}\\
&=||h||_2^2-||C_B^{*}M_{F}^{*}h||_2^2.
\end{align*}

Since the closure of $\{k_z\}_{z\in\DD^2}$ span $H^2(\DD^2)$, it follows that $||C_B^{*}M_{F}^*||\leq 1.$ Also we know from \eqref{F} that $F$ is bounded below. Thus
\begin{align*}
&||C_B||=||C_B^*||=||C_B^*M_{F}^*M_{1/{F}}^*||\leq  ||C_B^{*}M_{F}^*||\cdot||M_{1/{F}}^*||\\
\leq&||M_{1/{F}}^*||=||{1\over{F}}||_\infty =\left( \frac{1+|\phi(0)|}{1-|\phi(0)|} \right)^{1\over 2}\cdot \left( \frac{1+|\psi(0)|}{1-|\psi(0)|} \right)^{1\over 2}.
\end{align*}

\end{proof}

\begin{rem}
For an analytic self map $b$ of $\DD$, Jury in \cite{ju07} showed that if \beq\label{j1}\frac{1-\overline{b(w)}b(z)}{1-\bw z}\eeq is a positive kernel, then $C_b$ is bounded on $H^2(\DD)$. This condition is automatically satisfied since \eqref{j1} is the reproducing kernel of the de Branges-Rovnyak space associated with $b$ (\cite{deb-rov1}). Thus every composition operator on $H^2(\DD)$ is bounded. However, in general the positivity of $R$ is not a necessary condition for the boundedness of $C_B$. For example, let $B_r(z_1, z_2)=(rz_1, rz_1)$, and denote the corresponding function $R$ as $R_r$. Then it is easy to see that $C_{B_r}$ is bounded for every $r\in (0,1)$, and is unbounded for $r=1$.  If the functions $R_r$ were positive kernels for all $r<1$, then the pointwise limit $R_1$ would be a positive kernel as well, which would make $C_{B_1}$ bounded. It is a contradiction.

\end{rem}

\begin{rem}
A similar argument can be applied to some weighted Bergman spaces $L_a^\Ga(\DD^2)$, for $\Ga\in \ZZ^{+}$. These are reproducing kernel Hilbert spaces with reproducing kernels
$$k^\Ga(z,w)=\frac{1}{(1-\bar{w_1}z_1)^\Ga(1-\bar{w_2}z_2)^\Ga}.$$ If $\Ga=1$, then $L_a^\Ga(\DD^2)$ is just the Hardy space.
The corresponding sufficient condition for boundedness will be:
$$
(R(z,w))^\Ga=\left( \frac{1-\overline{\phi(w)}{\phi(z)}}{1-\bar{w_1}z_1}\cdot \frac{1-\overline{\psi(w)}{\psi(z)}}{1-\bar{w_2}z_2}\right)^\Ga
$$
is a positive kernel. If $R$ is a positive kernel and $\Ga$ is a positive integer, then $R^\Ga$ is positive as well (Proposition \ref{prop}(2)). Thus for $B$ satisfying the conditions in Theorem \ref{M}, $C_B$ is also bounded on $L^\Ga_a(\DD^2)$, for any positive interger $\Ga$.
\end{rem}

Since $||\phi||_\infty\leq 1$, it is a contractive multiplier of $H^2(\DD^2)$, by Theorem \ref{m}
$$\frac{1-\overline{\phi(w)}{\phi(z)}}{(1-\bar{w_1}z_1)(1-\bar{w_2}z_2)}$$ is a positive kernel, denoted as $k_w^\phi$. The Hilbert space with the above reproducing kernel, denoted by $\mathcal{H}(\phi)$, was introduced in \cite{abds}.  $\mathcal{H}(\phi)$ can be viewed as an analog of the de Branges-Rovnyak space on $\DD$ and is called a sub-Hardy Hilbert space of the bidisk (the terminology comes from the title of Sarason's book \cite{sar94}).

\begin{exm}\label{ex0}
We can easily show that $R$ is a positive kernel for some special cases, thus give some examples of bounded composition operators. These results were also obtained in \cite{sish}.
\begin{enumerate}
\item If one of $\phi, \psi$ is a constant (say $\psi$ is a constant with $|\psi|\leq 1$), then $R$ is a positive kernel because it is the positive kernel $\frac{1-\overline{\phi(w)}{\phi(z)}}{(1-\bar{w_1}z_1)(1-\bar{w_2}z_2)}$ multiplied by a positive constant.
\item If $\phi, \psi$ are one-variable functions in $z_1$ and $z_2$, respectively ($\phi=\phi(z_1)$, $\psi=\psi(z_2)$), then
$\frac{1-\overline{\phi(w_1)}\phi(z_1)}{1-\bar{w_1}z_1}$ and $\frac{1-\overline{\psi(w_2)}\psi(z_2)}{1-\bar{w_2}z_2}$ are positive kernels. By Proposition \ref{prop}, their product $R$ is a positive kernel.
\end{enumerate}
\end{exm}

In general, using Theorem \ref{m}, $R$ is a positive kernel if and only if $\psi$ is a contractive multiplier of $\mathcal{H}(\phi)$. So we obtained the following
\begin{cor}\label{cor}
Let $B=(\phi, \psi)$ be an analytic map from $\DD^2$ to $\DD^2$, where $\phi$ and $\psi$ are in $H^\infty(\DD^2)$ and bounded by $1$. If $\psi$ is a contractive multiplier of $\mathcal{H}(\phi)$, then $C_B$ is a bounded composition operator on $H^2(\DD^2)$. Moreover, $$||C_B||_{H^2\to H^2}\leq \left( \frac{1+|\phi(0)|}{1-|\phi(0)|} \right)^{1\over 2}\cdot \left( \frac{1+|\psi(0)|}{1-|\psi(0)|} \right)^{1\over 2}.$$
\end{cor}

In the next two sections, we study multipliers of $\mathcal{H}(\phi)$ and find nontrivial examples (other than those in Example \ref{ex0}) of bounded composition operators on $H^2(\DD^2)$ or $L_a^\Ga(\DD^2)$.

\section{Multipliers of sub-Hardy Hilbert spaces of the bidisk}
 For a bounded linear operator $A: H^2 \to H^2$, define the range space $$\mathcal{M}(A)=A H^2$$  and endow it with the inner product
$$\langle Af, Ag \rangle_{\mathcal{M}(A)}=\langle f, g \rangle_{H^2},\qq f,g\in H^2\ominus \m{Ker}A.$$
For $\phi\in H^\infty(\DD^2)$ with $||\phi||_\infty\leq 1$, let $A=(I-T_\phi T_{\overline{\phi}})^{1/2}$. It is easy to see that
\beq\label{b}
\mathcal{H}(\phi)=\mathcal{M}(A).
\eeq
For simplicity, we shall use $\mathcal{M}(\varphi)$ to denote $\mathcal{M}(T_\varphi)$, for $\varphi\in L^\infty(\TT^2)$.

It is easy to see that if $B$ maps $\DD^2$ into a compact subset of $\DD^2$, then $C_B$ is bounded on $H^2(\DD^2)$. In the rest of the note, we assume $||\phi||_\infty=1$. We first study the case when $\phi$ is an inner function.
The following theorem shows that if $\psi$ is a multiplier of $\mathcal{H}(\phi)$ for an inner function $\phi$, we only get the bounded composition operators in Example \ref{ex0}.

\begin{thm}
Let $\phi$ be a nonconstant inner function. If $\psi$ is a nonconstant multiplier of $\mathcal{H}(\phi)$, then one of the functions $\phi, \psi$ is a one-variable function in $z_1$ and the other is a one-variable function in $z_2$.
\end{thm}
\begin{proof}
By \cite{abds}*{Theorem 2.5}, $T_{\bar{z}_1}\phi\in \mathcal{H}(\phi)$ and then $$\psi T_{\bar{z}_1}\phi\in \mathcal{H}(\phi).$$
If $\phi$ is an inner function, then $\cH(\phi)=H^2(\DD^2)\ominus \phi H^2(\DD^2)$ is a closed subspace of $H^2(\DD^2)$. For every $f\in \cH(\phi)$, $g\in H^2(\DD^2)$, we see that
$$
\la T_{\bar{\phi}}f, g\ra=\la \bar{\phi} f, g\ra=\la f, \phi g\ra=0,
$$
which implies $$T_{\bar{\phi}}(\cH(\phi))=0.$$
So we have
\begin{align*}
0&=T_{\overline{\phi}} (\psi T_{\bar{z}_1}\phi )=T_{\overline{\phi}}\left(\psi\cdot \frac{\phi-\phi\circ P_1}{z_1}\right)\\
&=P(\bar{z}_1\psi(1-\overline{\phi}\cdot(\phi\circ P_1)))\\
&=T_{\bar{z}_1}\psi-P((\phi\circ P_1)\overline{\phi}\bar{z}_1\psi)\\
&=T_{\bar{z}_1}\psi-T_{(\phi\circ P_1)\overline{\phi}}T_{\bar{z}_1}\psi.
\end{align*}
The last equality holds because
\begin{align*}
P((\phi\circ P_1)\overline{\phi}\bar{z}_1\psi)&=P((\phi\circ P_1)\overline{\phi} \cdot(P(\bar{z}_1\psi)+(I-P)(\bar{z}_1\psi)))\\
&=T_{(\phi\circ P_1)\overline{\phi}}T_{\bar{z}_1}\psi+ P((\phi\circ P_1)\overline{\phi}\cdot (I-P)(\bar{z}_1\psi))\\
&=T_{(\phi\circ P_1)\overline{\phi}}T_{\bar{z}_1}\psi+ P((\phi\circ P_1)\overline{\phi}\cdot {\bz_1} (\psi\circ P_1))\\
&=T_{(\phi\circ P_1)\overline{\phi}}T_{\bar{z}_1}\psi.
\end{align*}
Thus
\begin{align}\label{p}
||T_{\bar{z}_1}\psi||_2&=||T_{(\phi\circ P_1)\overline{\phi}}T_{\bar{z}_1}\psi||_2=||P((\phi\circ P_1)\overline{\phi}\cdot T_{\bar{z}_1}\psi)||\\
\nnb&\leq ||(\phi\circ P_1)\overline{\phi}\cdot T_{\bar{z}_1}\psi||_2 \leq ||(\phi\circ P_1)\overline{\phi}||_\infty \cdot||T_{\bar{z}_1}\psi||_2.
\end{align}
Suppose $T_{\bar{z}_1}\psi\neq 0$, which means $\psi$ is not a one variable function in $z_2$. Then $||(\phi\circ P_1)\overline{\phi}||_\infty=1$ and equalities hold in \eqref{p}. We have
$$
T_{\bar{z}_1}\psi=P((\phi\circ P_1)\overline{\phi}\cdot T_{\bar{z}_1}\psi)=(\phi\circ P_1)\overline{\phi}\cdot T_{\bar{z}_1}\psi,
$$
which implies $\phi(z)=(\phi\circ P_1)(z)=\phi(0, z_2)$.

Therefore, one of the functions $\phi, \psi$ is a one-variable function in $z_2$. Applying the same argument for $T_{\bar{z}_2}\phi$, we know that one of the functions $\phi, \psi$ is a one-variable function in $z_1$ as well.
\end{proof}

Next, we assume there is a nonconstant function $a\in H^\infty(\DD^2)$ such that $$|a|^2+|\phi|^2=1$$ a.e. on $\TT^2$. We call $a$ the Pythagorean mate for $\phi$.
\begin{rem}
In the one variable case, a function $f \in H^2(\DD)$ has a Pythagorean mate if and only if $log(1-|f|^2)\in L^1(\TT)$, which means $f$ is not an extreme point of the unit ball of $H^\infty(\DD)$. The Pythagorean mate of $f$, if it exists, is an outer function and can be chosen uniquely so that it has positive value at $0$. However, the condition $log(1-|\phi|^2)\in L^1(\TT^2)$ is only necessary for $\phi\in H^2(\DD^2)$ to have a Pythagorean mate (see \cite{rud69}*{Chapter 3.5}) and a function in $H^2(\DD^2)$ may have a Pythagorean mate vanishing at $0$.
\end{rem}

Suppose $\phi$ has a Pythagorean mate $a$ and notice that $$T_{\ba} T_a=I- T_{\overline{\phi}}T_\phi \leq I-T_\phi T_{\overline{\phi}}.$$
By Douglas's Lemma (\cite{dou1}*{Theorem 1}) and \eqref{b}, $\mathcal{M}(\ba)\subset\mathcal{H}(\phi)$. In fact we have
\begin{lem}\label{5}
If $\phi$ is in $H^\infty(\DD^2)$ and has a Pythagorean mate $a$, then
$$
\mathcal{H}(\phi)\cap\mathcal{M}(\phi)=T_\phi \mathcal{M}(\ba).
$$
Moreover, every multiplier of $\mathcal{H}(\phi)$ is contained in $\mathcal{M}(\ba)$ and is a multiplier of $\mathcal{M}(\ba)$.
\end{lem}
\begin{proof}
The equality is a special case of the more general result \cite{sar94}*{I-9}. Suppose $f$ is a multiplier of $\mathcal{H}(\phi)$. Then
$$
 T_\phi T_f \mathcal{M}(\ba)=T_f T_\phi \mathcal{M}(\ba)=T_f\mathcal{H}(\phi)\cap T_f\mathcal{M}(\phi)\subset \mathcal{H}(\phi)\cap\mathcal{M}(\phi)=T_\phi \mathcal{M}(\ba),
$$
which implies $f$ is a multiplier of $\mathcal{M}(\ba)$.

Let $n=\min \{k+j : \la a, z_1^k z_2^j \ra_{H_2}\neq 0 \}$ ($n$ is the lowest degree of the non-vanishing terms in the Fourier expansion of $a$).
Pick a function $z_1^{k_0}z_2^{j_0}$ such that $k_0+j_0=n$ and  $\la a, z_1^k z_2^j \ra_{H_2}\neq 0$, we see that $T_{\ba} (z_1^{k_0}z_2^{j_0})=\ol{\la a, z_1^{k_0}z_2^{j_0} \ra}_{H_2}$. Thus $\mathcal{M}(\ba)$ contains constants and then every multiplier of $\mathcal{H}(\phi)$ is contained in $\mathcal{M}(\ba)$.
\end{proof}

\begin{lem}\label{6}
Suppose $\phi$ is in $H^\infty(\DD^2)$ and has a Pythagorean mate $a$. If $h$ is in  $H^2(\DD^2)$, then $h\in\mathcal{H}(\phi)$ if and only if $T_{\bar{\phi}} h\in \mathcal{M}(\ba)$.
\end{lem}
\begin{proof}
Again, this is a special case of the more general result \cite{sar94}*{I-8}.
\end{proof}

Lemma \ref{5} shows that multipliers of $\cH(\phi)$ are in the form $T_{\ba}h$, for some $h\in H^2(\DD^2)$. The next theorem gives a way to find multipliers of $\mathcal{H}(\phi)$. These conditions were first found in \cite{lotsar93} for the one-variable case. The proof we present here is essentially the same as in \cite{lotsar93}*{Theorem 2}.
\begin{thm}\label{con}
Suppose $\phi$ is in $H^\infty(\DD^2)$ and has a Pythagorean mate $a$. Let $\psi=T_{\ba}h$ for some $h\in H^2(\DD^2)$. If $H^*_{\bh} H_{\ba}$ and $H^*_{\bh} H_{\ol{\phi}}$ are bounded on $H^2(\DD^2)$, then $\psi$ is a multiplier of $\mathcal{H}(\phi)$.
\end{thm}
\begin{proof}
Let $f\in \cH(\phi)$. By Lemma \ref{6}, there exists a function $f^{+}\in H^2(\DD^2)$ such that $T_{\bar{\phi}}f=T_{\ba}f^{+}$, and it is sufficient to show $T_{\bar{\phi}}(\psi f)\in \cM(\ba)$.
Notice that for any $g\in [H^2(\DD^2)]^\perp$,
$$
H^*_{\bar{\psi}}g=P(\psi g)=P(g\cdot T_{\ba}h)=P(\ba gh)=T_{\ba}(gh)=T_{\ba}H^*_{\bh}g.
$$
The third equality holds because $T_g T_{\ba}=T_{g\ba}$.
Thus $H^*_{\bar{\psi}}=T_{\ba}H^*_{\bh}$. This together with \eqref{han} gives
\begin{align*}
T_{\bar{\phi}}(\psi f)&=T_{\psi\bar{\phi}}(f)=T_{\psi}T_{\bar{\phi}}f+H^*_{\bar{\psi}}H_{\bar{\phi}}f\\
&=T_{\psi}T_{\ba}f^{+}+H^*_{\bar{\psi}}H_{\bar{\phi}}f\\
&=T_{\psi\ba}f^{+}-H^*_{\bar{\psi}}H_{\ba}f^{+}+H^*_{\bar{\psi}}H_{\bar{\phi}}f\\
&=T_{\ba}(\psi f^{+})-T_{\ba}H^*_{\bh}H_{\ba}f^{+}+T_{\ba}H^*_{\bh}H_{\bar{\phi}}f.
\end{align*}
Since $H^*_{\bh} H_{\ba}$ and $H^*_{\bh} H_{\ol{\phi}}$ are bounded on $H^2(\DD^2)$, $H^*_{\bh} H_{\ba} f^{+}$ and $H^*_{\bh} H_{\ol{\phi}}f$ are in $H^2(\DD^2)$, which implies $T_{\bar{\phi}}(\psi f)\in \cM(\ba)$.
\end{proof}

Since $a$ and $\phi$ are bounded functions, $H^*_{\bh} H_{\ba}$ and $H^*_{\bh} H_{\ol{\phi}}$ are bounded if $H_{\bh}$ is bounded.
The converse is not true and it is an open problem (even for the one-variable case) to characterize the boundedness of the product of two Hankel operators.
If $\psi=T_{\ba}h$ for some function $h\in H^\infty(\DD^2)$, then $H^*_{\bh} H_{\ba}$ and $H^*_{\bh} H_{\ol{\phi}}$ are bounded.

\begin{thm}\label{pol}
Suppose $\phi$ is in $H^\infty(\DD^2)$ and has a Pythagorean mate $a$. Then every polynomial is a multiplier of $\cH(\phi)$.
\end{thm}
\begin{proof}
Consider the Fourier expansion of $a$. If the lowest degree in the Fourier expansion of $a$ is $n$, then $T_{\ba}$ maps polynomials of degree $n+N$ to polynomials of degree $N$.
Thus if $\psi$ is a polynomial, then $\psi=T_{\ba} h$ for some polynomial $h$. By Theorem \ref{con},  $\psi$ is a multiplier of $\cH(\phi)$.
\end{proof}

As a consequence of the above result, we present the following theorem regarding bounded composition operators.
\begin{thm}\label{com}
Suppose $\phi$ is in $H^\infty(\DD^2)$ and has a Pythagorean mate. Then for every non-zero polynomial $\psi$ bounded by $1$ on $\DD^2$, there exists a positive constant $k\leq \frac{1}{||\psi||_\infty}$ such that for every $c\in [0,k]$, the composition operator $C_B$ with $B=(\phi, c\psi)$ is a bounded on $H^2(\DD^2)$ and on $L^2_a(\DD^2)$.
\end{thm}
\begin{proof}
By Theorem \ref{pol}, $\psi$ is a multiplier of $\cH(\phi)$. Choose $k=\frac{1}{||M_{\psi}||}$, and $k\leq \frac{1}{||\psi||_\infty}$ by \eqref{norm}. The conclusion follows from Corollary \ref{cor}.
\end{proof}

\section{Examples}
In this section we present some concrete examples. The first example is $$\phi(z)=\frac{1+z_1}{2}$$ with Pythagorean mate $$a(z)=\frac{1-z_1}{2}.$$
We know from Theorem \ref{pol} that every polynomial is a multiplier of $\cH(\phi)$. In fact we can find more multipliers other than the polynomials.
\begin{thm}\label{ex1}
Let $\phi(z)=\frac{1+z_1}{2}$ and $a(z)=\frac{1-z_1}{2}$. Then $f\in\cM(\ba)$ if and only if there exist a function $f_1\in H^2(\DD^2)$ and a function $f_2\in H^2(\DD)$ such that
\beq\label{ex}
f(z)=(z_1-1)f_1(z)+f_2(z_2).
\eeq
If $f$ has the form \eqref{ex} and the function
\beq\label{ex'}g(z)=z_1f_1(z)+f_2(z_2)\eeq is bounded, then $f$ is a multiplier of $\cH(\phi)$.
\end{thm}
\begin{proof}
Suppose $f\in\cM(\ba)$. Then $f=T_{\ba}h={1\over 2}(h-T_{\bz_1}h)$, for some $h\in H^2(\DD^2)$. By formula \eqref{bs}, we have $$h=z_1 T_{\bz_1}h + h\circ P_1,$$ and then
$$
f={1\over 2}(z_1 T_{\bz_1}h + h\circ P_1-T_{\bz_1}h)=(z_1-1)({1\over 2}\cdot T_{\bz_1}h)+ {1\over 2}(h\circ P_1).
$$
Notice that $h\circ P_1$ is a one variable function in $z_2$, so we get the desired form \eqref{ex}.

Conversely, if $f$ has the form \eqref{ex3}, then $h$ is determined by $$h(z)=2z_1f_1(z)+2f_2(z_2),$$ and $f=T_{\ba}h$.
By Theorem \ref{con}, we see that functions $f$ having the form \eqref{ex} with $||g||_\infty<\infty$ are multipliers of $\cH(\phi)$ as well.
\end{proof}

A similar argument can be applied to $\phi(z)=\frac{1+z_1z_2}{2}$, where the Pythagorean mate of $\phi$ is $a(z)=\frac{1-z_1z_2}{2}$. We can obtain a result similar to Theorem \ref{ex1}.
\begin{thm}\label{ex2}
Let $\phi(z)=\frac{1+z_1z_2}{2}$ and $a(z)=\frac{1-z_1z_2}{2}$. Then $f\in\cM(\ba)$ if and only if there exist functions $f_1, f_2\in H^2(\DD)$ and a function $f_3\in H^2(\DD^2)$ such that
\beq\label{ex2}
f(z)=f_1(z_1)+f_2(z_2)+(z_1z_2-1)f_3(z).
\eeq
If $f$ has the form \eqref{ex2} and the function \beq\label{ex2'}g(z)=f_1(z_1)+f_2(z_2)+z_1z_2f_3(z)\eeq is bounded, then $f$ is a multiplier of $\cH(\phi)$.
\end{thm}
\begin{proof}
For every $h\in H^2(\DD^2)$, notice that
$$
h=z_1z_2 T_{\bz_1\bz_2}h + h\circ P_1 +h\circ P_2-h(0,0).
$$
The rest of the proof is similar to that in Theorem \ref{ex1}.
\end{proof}

In the next example we let $\phi(z)=\frac{z_1+z_2}{2}$. Then $\phi$ has a Pythagorean mate $a(z)=\frac{z_1-z_2}{2}$, which vanishes at $0$.
\begin{thm}
Let $\phi(z)=\frac{z_1+z_2}{2}$ and $a(z)=\frac{z_1-z_2}{2}$. Then $f\in\cM(\ba)$ if and only if there exist functions $f_1, f_2\in H^2(\DD)$ and a function $f_3\in H^2(\DD^2)$ such that
\beq\label{ex3}
f(z)=f_1(z_1)+f_2(z_2)+(z_1-z_2)f_3(z).
\eeq
If $f$ has the form \eqref{ex3} and the function \beq\label{ex3'}g(z)=z_1f_1(z_1)-z_2f_2(z_2)+z_1z_2f_3(z)\eeq is bounded, then $f$ is a multiplier of $\cH(\phi)$.
\end{thm}
\begin{proof}
Suppose $f\in\cM(\ba)$. Then $f=T_{\ba}h={1\over 2}(T_{\bz_1}h-T_{\bz_2}h)$, for some $h\in H^2(\DD^2)$. Let
$$h(z)=h_0+z_1h_1(z_1)+z_2h_2(z_2)+z_1z_2h_3(z),$$
where $h_0$ is constant, $h_1, h_2$ are in $H^2(\DD)$ and $h_3$ is in $H^2(\DD^2)$.
By formula \eqref{bs}, we have $$f(z)={1\over 2}(T_{\bz_1}h-T_{\bz_2}h)(z)={1\over 2}(h_2(z_2)-h_1(z_1))+{1\over 2}(z_2-z_1)h_3(z),$$
which is in the form \eqref{ex3}.

Conversely, if $f$ has the form \eqref{ex3}, then we can pick $h$ as
$$h(z)=-2z_1f_1(z_1)+2z_2f_2(z_2)-2z_1z_2f_3(z),$$ and $f=T_{\ba}h$.
By Theorem \ref{con}, we see that functions $f$ having the form \eqref{ex} with $||g||_\infty<\infty$ is a multipliers of $\cH(\phi)$ as well.
\end{proof}

One can get the corresponding bounded composition operators using Corollary \ref{cor}. We summarize the examples given above as follows.
\begin{thm}
Let $B=(\phi, \psi)$ be an analytic map from $\DD^2$ to $\DD^2$, and let $f$ be a bounded analytic function on $\DD^2$. Then the composition operator $C_B$ is bounded on $H^2(\DD^2)$ in each of the following cases:
\begin{enumerate}
\item $\phi(z)=\frac{1+z_1}{2}$, $\psi(z)=cf$ for some constant $c\leq \frac{1}{||f||_\infty}$ , where $f$ has the form \eqref{ex} and the function \eqref{ex'} is bounded.
\item $\phi(z)=\frac{1+z_1z_2}{2}$,$\psi(z)=cf$ for some constant $c\leq \frac{1}{||f||_\infty}$, where $f$ has the form \eqref{ex2}, and the function \eqref{ex2'} is bounded.
\item $\phi(z)=\frac{z_1+z_2}{2}$, $\psi(z)=cf$ for some constant $c\leq \frac{1}{||f||_\infty}$, where $f$ has the form \eqref{ex3}, and the function \eqref{ex3'} is bounded.
\end{enumerate}
\end{thm}

\bibliography{references}
\end{document}